\documentclass[12pt]{amsart}
\usepackage[fullpage]{mymacros}
\newcommand{\ts}{T}
\newcommand{\tdel}{\widetilde{\partial}}
\newcommand{\tschub}{\widetilde{\mathfrak S}}
\newcommand{\schub}{\mathfrak S}

\title{Twisted Schubert polynomials}
\author{Ricky Ini Liu}
\address{Ricky Ini Liu, Department of Mathematics, North Carolina State University, Raleigh, NC 27695}
\email{riliu@ncsu.edu}
\thanks{The author is partially supported by a National Science Foundation Grant (DMS 1758187).}

\begin{document}
	\maketitle

	\begin{abstract}
		We prove that twisted versions of Schubert polynomials defined by $\tschub_{w_0} = x_1^{n-1}x_2^{n-2} \cdots x_{n-1}$ and $\tschub_{ws_i} = (s_i+\del_i)\tschub_w$ are monomial positive and give a combinatorial formula for their coefficients. In doing so, we reprove and extend a previous result about positivity of skew divided difference operators and show how it implies the Pieri rule for Schubert polynomials. We also give positive formulas for double versions of the $\tschub_w$ as well as their localizations.
	\end{abstract}

\section{Introduction}

The operators $T_i = s_i + \del_i$ on the polynomial ring $\CC[x_1, \dots, x_n]$ (where $s_i$ switches $x_i$ and $x_{i+1}$, and $\del_i = \frac{1-s_i}{x_i-x_{i+1}}$ is the divided difference operator) satisfy the Coxeter relations and therefore define a twisted action of the symmetric group on polynomials. The action of these operators, particularly on the quotient $\CC[x_1, \dots, x_n]/I$ where $I$ is the ideal generated by symmetric polynomials with no constant term, has been studied previously due to its relation to: the Chern-Schwartz-MacPherson (CSM) classes of Schubert cells in flag varieties \cite{AluffiMihalcea, AMSS}, Maulik-Okounkov stable envelopes \cite{MaulikOkounkov, RimanyiVarchenko}, and Demazure-Lusztig operators and Hecke algebras \cite{LLT, LLT2}. Similar operators have also been considered in the context of generalized Schubert, key, and Grothendieck polynomials \cite{Kirillov2}. 

In this paper, we consider a twisted analogue of Schubert polynomials defined for permutations $w \in S_n$ by $\tschub_{w_0} = x_1^{n-1}x_2^{n-2} \cdots x_{n-1}$ for $w=w_0 = n \cdots 321$, and $\tschub_{ws_i} = T_i \tschub_w$. (An important note: we do not, as in many of the references above, consider these as classes modulo $I$ but instead as polynomials in their own right.) The minimum degree part of $\tschub_w$ is the usual Schubert polynomial $\schub_w$. It is well known \cite{ BergeronBilley, BilleyJockuschStanley, FominStanley} that Schubert polynomials are monomial positive. Although the operators $T_i$ do not in general preserve monomial positivity, our main result is that the polynomials $\tschub_w$ are always monomial positive, and we give a combinatorial formula for their coefficients in Theorem~\ref{thm:positive} in terms of certain chains in Bruhat order.

Our proof of Theorem~\ref{thm:positive} is entirely algebraic and is closely related to the study of skew divided difference operators (see \cite{Kirillov, Liudd, Macdonald}), which arise when applying the twisted Leibniz rule for ordinary divided difference operators. It was shown in \cite{Liudd} that the skew divided difference operators can always be expressed in terms of the usual divided difference operators $\del_{ij}$ for $i<j$ with positive coefficients. We extend and give an alternate proof of that result here, and as an application, we demonstrate how this result can be used to derive the Pieri rule for Schubert polynomials.

We also define a ``double version'' of the twisted Schubert polynomials $\tschub_w(x,y)$ in two sets of variables (which, up to signs, correspond to equivariant CSM classes of Schubert cells). Our combinatorial formula naturally extends to this setting, and we also give a positive formula for the localizations $\tschub_v(y_{w(1)}, \dots, y_{w(n)}; y_1, \dots, y_n)$. (An equivalent formula for these localizations can also be found in \cite{AMSS, Su}.)

We begin by reviewing background on the symmetric group, divided difference operators, and Schubert polynomials in \S2. In \S3, we study the twisted operators $T_i$ and relate them to skew divided difference operators. In particular, we prove that they exhibit certain positivity properties, and we compute their actions on elementary and complete homogeneous symmetric polynomials. In \S4, we apply the results of the previous section to prove that $\tschub_w$ are monomial positive and give a combinatorial formula for their coefficients. We also give a formula for the double polynomials $\tschub_w(x,y)$ and their localizations. We conclude in \S5 with some closing remarks.

\section{Background}
In this section, we give some background about the symmetric group, divided difference operators, and Schubert polynomials.
\subsection{The symmetric group} \label{sec:symmetric}
Let $S_n$ denote the \emph{symmetric group} on $[n]$. We will write $s_{ij}$ for the transposition switching $i$ and $j$, and we will abbreviate $s_i = s_{i,i+1}$ for the \emph{simple transposition} switching $i$ and $i+1$. We denote by $w_0 \in S_n$ the permutation $n \cdots 321$.

For any $w \in S_n$, a \emph{reduced expression} or \emph{reduced word} is an expression $s_{i_1}s_{i_2} \cdots s_{i_{\ell(w)}}$ for $w$ as a product of simple transpositions of minimal length $\ell(w)$. Any two reduced expressions for $w$ can be transformed into one another by applying a sequence of \emph{Coxeter relations} of the form $s_is_{i+1}s_i = s_{i+1}s_is_{i+1}$ and $s_is_j = s_js_i$ for $|i-j| > 1$. (In other words, the third Coxeter relation $s_i^2 = 1$ is not needed.)

A product of simple transpositions can be visualized in terms of a \emph{wiring diagram} consisting of $n$ wires passing from left to right, where an occurrence of $s_i$ indicates that the $i$th and $(i+1)$st wires from the top should switch places. A reduced word is one whose wiring diagram has no two wires crossing more than once.

For any reduced expression $w=s_{i_1}s_{i_2} \cdots s_{i_\ell}$, we can associate to each $s_{i_m}$ a pair $(\alpha_m, \beta_m)$ by 
\[s_{i_m} \cdot s_{i_{m+1}} \cdots s_{i_\ell} = s_{i_{m+1}} \cdots s_{i_\ell} \cdot s_{\alpha_m\beta_m},\tag{$*$}\]
with $\alpha_m < \beta_m$. Explicitly,
\begin{align*}
\alpha_m &= s_{i_\ell}s_{i_{\ell-1}} \cdots s_{i_{m+1}}(i_m),\\
\beta_m &= s_{i_\ell}s_{i_{\ell-1}} \cdots s_{i_{m+1}}(i_{m}+1).
\end{align*}
In terms of the wiring diagram, number the wires $1, \dots, n$ on the right. At the $m$th crossing (counting from the left), the two wires that cross are $\alpha_m$ and $\beta_m$. From this description, it is easy to see that the pairs $(\alpha_m, \beta_m)$ are all distinct, ranging over all pairs $\alpha < \beta$ such that $w(\alpha) > w(\beta)$.

\begin{ex}
The values of $\alpha_m$ and $\beta_m$ for  $w = s_1s_2s_3s_1s_2s_1$ are shown below.
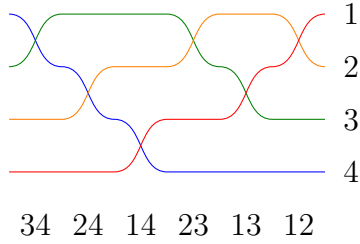
\[
\begin{tikzpicture}[scale=.7]
\draw[blue] (0,4) to[out = 0, in=180] (1,3) to[out = 0, in=180] (2,2) to[out = 0, in=180] (3,1)--(6,1);
\draw[green!50!black] (0,3) to[out = 0, in=180] (1,4) -- (3,4) to[out = 0, in=180] (4,3) to[out = 0, in=180](5,2)--(6,2);
\draw[orange] (0,2) --(1,2) to[out = 0, in=180] (2,3)--(3,3) to [out = 0, in=180] (4,4)--(5,4) to[out = 0, in=180] (6,3);
\draw[red] (0,1)--(2,1) to[out = 0, in=180] (3,2)--(4,2) to[out = 0, in=180] (5,3) to[out = 0, in=180] (6,4);
\node at (6.5,4) {$1$};
\node at (6.5,3) {$2$};
\node at (6.5,2) {$3$};
\node at (6.5,1) {$4$};
\node at (.5,0) {$34$};
\node at (1.5,0) {$24$};
\node at (2.5,0) {$14$};
\node at (3.5,0) {$23$};
\node at (4.5,0) {$13$};
\node at (5.5,0) {$12$};
\end{tikzpicture}
\]
\end{ex}


We may use $\alpha_m$ and $\beta_m$ to compute the change in $w$ upon removing some transpositions from its reduced expression. 

\begin{prop} \label{prop:factor}
Let $w = s_{i_1}s_{i_2} \cdots s_{i_\ell}$ be a reduced expression with $\alpha_j$ and $\beta_j$ defined as above. Let $J \subset [\ell]$ be any subset. Then 
\[v = \prod_{j \in J} s_{i_j} = w \cdot \prod_{j \notin J} s_{\alpha_j \beta_j}.\]
\end{prop}
(Here and elsewhere, products such as $\prod_j s_{i_j}$ are taken from left to right in increasing order of index $j$.)
\begin{proof}
By equation $(*)$, $w s_{\alpha_j\beta_j} = s_{i_1}s_{i_2} \cdots \widehat{s_{i_j}}\cdots s_{i_\ell}$. Iterating over all $j \notin J$ (in increasing order) gives the result.
\end{proof}

The \emph{(strong) Bruhat order} on $S_n$ is defined such that if $w = s_{i_1}s_{i_2} \cdots s_{i_\ell}$ is a reduced expression, then $v < w$ in Bruhat order if $v = \prod_{j \in J} s_{i_j}$ for some subset $J \subset [\ell]$. (This definition does not depend on the choice of reduced expression for $w$.) The cover relations in Bruhat order are given by $v \lessdot v s_{ab}$, where $\ell(vs_{ab}) = \ell(v)+1$.

For more information on the symmetric group, see for instance \cite{BjornerBrenti}.

\subsection{Divided difference operators} \label{sec:dd}

The symmetric group $S_n$ acts on the polynomial ring $\CC[x_1, \dots, x_n]$ by permuting the variables: $(wf)(x_1, \dots, x_n) = f(x_{w(1)}, \dots, x_{w(n)})$. For any $1 \leq i < j \leq n$, we define the \emph{divided difference operator}
\[\partial_{ij} = \frac{1-s_{ij}}{x_i-x_j}.\]
If $j=i+1$, then we write $\partial_i = \partial_{i,i+1}$ for the \emph{simple divided difference operators}. 

It is straightforward to verify the following proposition.
\begin{prop} \label{prop:relations}
The divided difference operators satisfy the following relations for distinct $i$, $j$, $k$, and $l$:
\begin{align}
\label{negative} \del_{ij} &= -\del_{ji},\\
\label{nil} \del_{ij}^2 &= 0,\\
\label{nonadjacent} \del_{ij}\del_{kl} &= \del_{kl}\del_{ij},\\
\label{triangle} \del_{ij} \del_{jk} &= \del_{ik}\del_{ij} + \del_{jk}\del_{ik},\\
\label{braid} \del_{ij}\del_{jk}\del_{ij} &= \del_{jk}\del_{ij}\del_{jk},\\
\label{commute} \del_{ij}w &= w \del_{w^{-1}(i)w^{-1}(j)} &&\text{ for all } w \in S_n,\\
\label{leibniz} \del_{ij}(PQ) &= \del_{ij}(P) \cdot Q + s_{ij}(P) \cdot \del_{ij}(Q) &&\text{ for all } P,Q \in \CC[x_1, \dots, x_n].
\end{align}
\end{prop}

If $w$ has reduced expression $s_{i_1}s_{i_2} \cdots s_{i_{\ell(w)}}$, then we define $\partial_w = \partial_{i_1}\partial_{i_2} \cdots \partial_{i_{\ell(w)}}$. This does not depend on the choice of reduced expression since the $\del_i$ satisfy the \emph{nil-Coxeter relations} \eqref{nil}, \eqref{nonadjacent}, and \eqref{braid}.

In \cite{Macdonald}, Macdonald defines the \emph{skew divided difference operators} $\del_{w/v}$ for $v,w \in S_n$ such that, for any $P,Q \in \CC[x_1, \dots, x_n]$,
\[\del_w(PQ) = \sum_{v \in S_n} v(\del_{w/v}P) \cdot \del_vQ.\]
One can compute $\del_{w/v}$ by applying the Leibniz rule \eqref{leibniz} repeatedly for $\del_w = \del_{i_1}\del_{i_2} \cdots \del_{i_\ell}$ and then using relation \eqref{commute} to move all of the elements of $S_n$ to the left.

Explicitly, for any subset $J \subset [\ell]$, let $\varphi_J = \prod_{j=1}^\ell \varphi_j(J)$, where $\varphi_j(J) = s_{i_j}$ if $j \in J$ and $\del_{i_j}$ if $j \notin J$. Then
\[\del_{w/v} = v^{-1} \sum_J \varphi_J,\]
where $J$ ranges over all subsets of $[\ell]$ for which $\prod_{j \in J} s_{i_j}$ is a reduced expression for $v$. The value of $\del_{w/v}$ will not depend on the choice of reduced word for $w$. Clearly $\del_{w/v}=0$ unless $v<w$ in Bruhat order.

\begin{ex} \label{ex:12321a}
	Let $w = s_1s_2s_3s_2s_1$ and $v=s_1$. Then we may take either $J=\{1\}$ or $\{5\}$, which gives
	\begin{align*}
	\del_{w/v} &= s_1^{-1} (s_1 \del_2\del_3\del_2\del_1 + \del_1\del_2\del_3\del_2s_1)\\
	&=\del_{23}\del_{34}\del_{23}\del_{12} + \del_{21}\del_{13}\del_{34}\del_{13}\\
	&=\del_{23}\del_{34}\del_{23}\del_{12} - \del_{12}\del_{13}\del_{34}\del_{13}.
	\end{align*}
\end{ex}

The previous example shows that the naive expansion of $\del_{w/v}$ when expressed in terms of $\del_{ij}$ for $i<j$ may contain negative coefficients. However, it was proved in \cite{Liudd} that one can always rewrite it in a form that has only positive coefficients. (See Example~\ref{ex:12321pos} below for how to do this for the case in Example~\ref{ex:12321a}.) We will generalize this positivity result in Theorem~\ref{thm:tdel} below.

\subsection{Schubert polynomials}

For any permutation $w \in S_n$, the \emph{Schubert polynomial} $\schub_w$ is defined by 
\[\schub_w = \del_{w^{-1}w_0} (x_1^{n-1} x_2^{n-2} \cdots x_{n-1}).\]
In other words,
$\schub_{w_0} = x_1^{n-1}x_2^{n-2} \cdots x_{n-1}$, and $\schub_{ws_i} = \del_i \schub_w$ if $\ell(ws_i) = \ell(w)-1$.

Schubert polynomials exhibit a \emph{stability} property in that $\schub_w$ is unchanged under the embedding $S_n \to S_{n+1}$ in which $S_n$ acts on the first $n$ letters of $[n+1]$. Hence it is often natural to instead define Schubert polynomials $\schub_w$ for $w \in S_\infty$, that is, when $w$ is a permutation of the positive integers that fixes all but finitely many elements. In this context, the Schubert polynomials $\schub_w$ for $w \in S_\infty$ form a basis for the polynomial ring $\CC[x_1, x_2, \dots]$.

It is well known that the expansion of a Schubert polynomial in the monomial basis always has nonnegative coefficients. One common description is as follows. A \emph{pipe dream} or \emph{rc-graph} is a type of wiring diagram in which each box $(i,j)$ with $i,j \geq 1$ is filled with either a cross or a pair of elbows. Such a diagram corresponds to a permutation $w$ if the wire entering in row $i$ exits in column $w(i)$. A pipe dream is called \emph{reduced} if no two pipes cross more than once. The weight of a pipe dream is $x_{i_1}x_{i_2} \cdots$, where $i_1$, $i_2$, \dots, are the rows containing the crosses of the pipe dream. Then $\schub_w$ is the sum of the weights of all reduced pipe dreams corresponding to $w$. See \cite{BergeronBilley, BilleyJockuschStanley, FominStanley} for more details.

\begin{example} \label{ex:pipe}
	The following are the two reduced pipe dreams for the permutation $w=2431$. Hence $\schub_w = x_1^2x_2x_3 + x_1x_2^2x_3$.
	
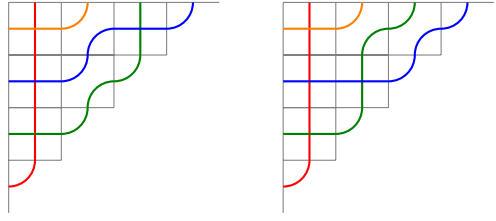
\[
	\begin{tikzpicture}[scale=.7]
	\draw[black!50] (0,0) grid (3,-1) (0,-1) grid (2,-2) (0,-2) grid (1,-3) (3,0)--(4,0) (0,-3)--(0,-4);
	\draw[red,thick] (.5,0)--(.5,-3) to[out = -90, in=0] (0,-3.5);
	\draw[orange,thick] (1.5,0) to[out=-90,in=0] (1,-.5)--(0,-.5);
	\draw[green!50!black,thick] (2.5,0)--(2.5,-1) to[out=-90,in=0] (2,-1.5) to[out=180,in=90](1.5,-2) to [out=-90,in=0] (1,-2.5) to (0,-2.5);
	\draw[blue,thick] (3.5,0) to [out=-90,in=0] (3,-.5)--(2,-.5) to [out=180, in=90] (1.5,-1) to[out=-90,in=0] (1,-1.5)--(0,-1.5);
	\end{tikzpicture}
	\qquad
	\begin{tikzpicture}[scale=.7]
	\draw[black!50] (0,0) grid (3,-1) (0,-1) grid (2,-2) (0,-2) grid (1,-3) (3,0)--(4,0) (0,-3)--(0,-4);
	\draw[red,thick] (.5,0)--(.5,-3) to[out = -90, in=0] (0,-3.5);
	\draw[orange,thick] (1.5,0) to[out=-90,in=0] (1,-.5)--(0,-.5);
	\draw[green!50!black,thick] (2.5,0) to[out=-90,in=0] (2,-.5) to [out=180, in=90] (1.5,-1)--(1.5,-2) to [out=-90,in=0] (1,-2.5) to (0,-2.5);
	\draw[blue,thick] (3.5,0) to [out=-90,in=0] (3,-.5) to[out=180,in=90] (2.5,-1)to[out=-90,in=0] (2,-1.5)--(0,-1.5);
	\end{tikzpicture}
	\]
	
\end{example}

The Schubert polynomials arise in the study of the flag variety $\mathscr F_n = GL_n(\CC)/B$ (where $B$ is the subgroup of upper triangular matrices). Specifically, they are polynomial representatives of the classes of Schubert varieties in the cohomology ring $H^*(\mathscr F_n) = H^*(\mathscr F_n; \CC) = \CC[x_1, \dots, x_n]/I$, where $I$ is the ideal generated by symmetric polynomials in $x_1, \dots, x_n$ with no constant term. Since multiplication in $H^*(\mathscr F_n)$ corresponds to intersection of Schubert varieties, one can deduce that in the expansion
\[\schub_u \cdot \schub_v = \sum_w c_{uv}^w \schub_w,\]
the \emph{generalized Littlewood-Richardson coefficients} (or \emph{Schubert structure constants}) $c_{uv}^w$ are always nonnegative integers. It is an important open problem to give a combinatorial description of the coefficients $c_{uv}^w$---see \cite{Buch, Coskun} for some partial progress.

For instance, in the special case that $\schub_u$ is an elementary or complete homogeneous symmetric polynomial, we have the following \emph{Pieri rule} for Schubert polynomials (see, for instance, \cite{LascouxSchutzenberger, Postnikov, Sottile}).
\begin{theorem}\label{thm:pieri}
	Let $v \in S_n$, and let $e_m^{(k)}$ be the $m$th elementary symmetric polynomial in $x_1, \dots, x_k$. Then
	\[\schub_v \cdot e_m^{(k)} = \sum_w \schub_w,\]
	where $w \in S_n$ ranges over all permutations such that there exists a sequence
	\[v \lessdot v s_{a_1b_1} \lessdot v s_{a_1b_1}s_{a_2b_2} \lessdot \cdots \lessdot v s_{a_1b_1}s_{a_2b_2} \cdots s_{a_tb_t} = w,\]
	where $a_i \leq k < b_i$ for all $i$, and the $a_i$ are distinct.
	
	Similarly, if one replaces $e_m^{(k)}$ with $h_m^{(k)}$, the $m$th complete homogeneous symmetric polynomial in $x_1, \dots, x_k$, then the same result holds except that instead the $b_i$ are distinct.
\end{theorem} 

We will deduce this Pieri rule using the action of skew divided difference operators in Corollary~\ref{cor:pieri} below.

\medskip

The definition of Schubert polynomials easily implies that for $v,w \in S_n$,
\[
\del_v \schub_w = \begin{cases}
\schub_{wv^{-1}}&\text{if } \ell(wv^{-1}) = \ell(w) - \ell(v),\\
0&\text{otherwise.}
\end{cases}
\]
In particular, if $\ell(v) = \ell(w)$, then $\del_v \schub_w = 1$ if $v=w$, and $0$ otherwise. Using this, one can deduce the following result from Macdonald \cite{Macdonald}.
\begin{prop}[\cite{Macdonald}] \label{prop:cuvw}
	Let $u,v,w \in S_n$ such that $\ell(u) + \ell(v) = \ell(w)$. Then $\del_{w/v} \schub_u = c_{uv}^w$.
\end{prop}
\begin{proof}
	By the discussion above, the coefficient $c_{uv}^w$ is equal to
	\[\del_w(\schub_u\schub_v) = \sum_{v'} v'(\del_{w/v'}\schub_u) \cdot \del_{v'}\schub_v.\]
	By degree considerations, the only nonzero terms can arise when $\ell(v) = \ell(v')$, in which case we must have $v = v'$. Thus the right hand side simplifies to $\del_{w/v} \schub_u$, as desired.
\end{proof}

For more information regarding skew divided difference operators and Schubert polynomials, see \cite{Kirillov, Liudd, Macdonald}.

\medskip

One can also define the \emph{double Schubert polynomials} in two sets of variables $x_1, \dots, x_n$ and $y_1, \dots, y_n$ by
\[\schub_w(x, y) = \del_{w^{-1}w_0}\prod_{i+j \leq n} (x_i-y_j),\]
where the divided difference operators act on the $x$-variables but not the $y$-variables.

Just as $\schub_w$ is monomial positive, $\schub_w(x,y)$ is a polynomial in $x_i-y_j$ with positive coefficients. In fact, the combinatorial description of $\schub_w$ in terms of pipe dreams extends to $\schub_w(x,y)$ by weighting a cross in row $i$ and column $j$ by $x_i-y_j$.

The double Schubert polynomials represent the classes of Schubert varieties in the equivariant cohomology ring $H^*_{T}(\mathscr F_n)$, where $T = (\CC^*)^n$ is the $n$-dimensional torus.
According to GKM theory \cite{GKM}, there is an injective map \[H_T^*(\mathscr F_n) \to \bigoplus_{w \in S_n} H^*_T(e_w) = \bigoplus_{w \in S_n} \CC[y_1, \dots, y_n],\]
where $e_w$ is the $T$-fixed point corresponding to $w \in S_n$. We define the \emph{localization} of $\schub_v$ at $w$ to be the specialization
\[\schub_v(wy,y) = \schub_v(y_{w(1)}, \dots, y_{w(n)}; y_1, \dots, y_n).\]
Then the localization of $\schub_v$ at $w$ is the image of the Schubert class in $H^*_T(e_w)$. 

The following formula (sometimes called Billey's formula \cite{Billey}) gives a combinatorial expression for these localizations.

\begin{thm}[\cite{Billey}] \label{thm:billey}
Let $w^{-1} = s_{i_1}s_{i_2} \cdots s_{i_\ell}$ be a reduced expression, and define $\alpha_j$ and $\beta_j$ as in $(*)$. Then
\[\schub_v(wy,y) = \sum_{J} \prod_{j \in J} (y_{\beta_j}-y_{\alpha_j}),\]
where $J$ ranges over all subsets of $[\ell]$ for which $\prod_{j \in J}s_{i_j}$ is a reduced word for $v^{-1}$.
\end{thm}
Note that $\schub_v(wy,y)$ is a polynomial in $y_b-y_a$, $b>a$, with positive coefficients. We will give a generalization of this formula in Theorem~\ref{thm:local} below.

\section{Twisted operators}
In this section, we introduce a twisted version of divided difference operators and discuss some of their properties. In particular, we relate them to operators $\tdel_{w/v}$ which are closely related to skew divided difference opearators. We will prove that these $\tdel_{w/v}$ can be expressed positively in terms of $\del_{ij}$, $i<j$. We will then compute their action explicitly on elementary and complete homogeneous symmetric polynomials $e_m(x_1, \dots, x_k)$ and $h_m(x_1, \dots, x_k)$, which will imply the Pieri rule for Schubert polynomials.

\subsection{Definitions}
Define \emph{twisted operators}
\[\ts_i = s_i + \partial_i.\]
For any expression $w = s_{i_1} \cdots s_{i_{\ell}}$ (not necessarily reduced!), one can expand $T_w=T_{i_1} \cdots T_{i_\ell}$ to obtain an expression in terms of divided difference operators and the action of $S_n$. By moving all the elements of $S_n$ to the left using relation \eqref{commute}, we can write
\[T_w = \sum_{v \in S_n} v\tdel_{w/v}\]
for some operators $\tdel_{w/v}$ in the algebra generated by the $\del_{ij}$.

As in Section~\ref{sec:dd}, we can give an explicit formula for $\tdel_{w/v}$:
\[\tdel_{w/v} = v^{-1} \sum_J \varphi_J,\]
where $J$ ranges over all subsets of $[\ell]$ for which $\prod_{j \in J} s_{i_j} = v$, where this expression is not necessarily reduced. (Recall that $\varphi_J = \prod_{j=1}^{\ell} \varphi_j(J)$, where $\varphi_j(J) = s_{i_j}$ for $j \in J$ and $\del_{i_j}$ for $j \notin J$.) As we will see, $\tdel_{w/v}$ will not depend on the initial choice of expression for $w$.

Note that if we start with a reduced expression for $w$, then the only difference between the definitions of $\tdel_{w/v}$ and $\del_{w/v}$ is that $\prod_{j \in J}s_{i_j}$ must be a reduced expression for $v$ in $\del_{w/v}$ but not in $\tdel_{w/v}$. It follows that $\del_{w/v}$ is the maximum degree part of $\tdel_{w/v}$.


\begin{ex} \label{ex:12321}
	Let $w = s_1s_2s_3s_2s_1$ and $v=s_1$ as in Example~\ref{ex:12321a}. Then we may take either $J=\{1\}$, $\{5\}$, $\{1,2,4\}$, or $\{2,4,5\}$, which gives
	\begin{align*}
	\tdel_{w/v} &= s_1^{-1} (s_1 \del_2\del_3\del_2\del_1 + \del_1\del_2\del_3\del_2s_1 + s_1s_2 \del_3s_2\del_1 + \del_1s_2\del_3s_2s_1)\\
	&=\del_{23}\del_{34}\del_{23}\del_{12} + \del_{21}\del_{13}\del_{34}\del_{13} + \del_{24}\del_{12} + \del_{21}\del_{14}\\
	&= \del_{23}\del_{34}\del_{23}\del_{12} - \del_{12}\del_{13}\del_{34}\del_{13} + \del_{24}\del_{12} - \del_{12}\del_{14}.
	\end{align*}
\end{ex}

The operators $\tdel_{w/v}$ are well-defined due to the following proposition.


\begin{prop}
	The operators $T_i$ satisfy the Coxeter relations $T_i^2 = 1$, $T_iT_j = T_jT_i$ for $|i-j| > 1$, and $T_iT_{i+1}T_i = T_{i+1}T_iT_{i+1}$. The operators $T_w$ and $\tdel_{w/v}$ do not depend on the choice of (not necessarily reduced) expression $w = s_{i_1} \cdots s_{i_\ell}$.
\end{prop}
\begin{proof}
	Any two expressions for $w$ can be obtained from one another by repeatedly applying Coxeter relations to some contiguous subexpression. By fixing the subset $J$ outside this subexpression, we find that it suffices to show that $T_w$ and $\tdel_{w/v}$ are well-defined whenever $w$ appears in some Coxeter relation. This follows from a straightforward calculation. For example, if $w = s_is_{i+1}s_i$ and $v=s_i$, then
	\[\tdel_{w/v} = s_i^{-1}(s_i\del_{i+1}\del_i + \del_i\del_{i+1}s_i) = \del_{i+1,i+2}\del_{i,i+1}+\del_{i+1,i}\del_{i,i+2},\]
	whereas if $w=s_{i+1}s_is_{i+1}$, then
	\[\tdel_{w/v} = s_i^{-1}(\del_{i+1}s_i\del_{i+1}) = \del_{i,i+2}\del_{i+1,i+2}.\]
	But these two expressions can be equated using relations \eqref{negative} and \eqref{triangle}. Similarly, if $v = id$, then in both cases, $\tdel_{w/v} = \del_w + \del_{i,i+2}$. The other cases follow similarly.
\end{proof}


The operators $T_w$ satisfy the following Leibniz rule.

\begin{prop} \label{prop:leibniz}
	Let $P,Q \in \CC[x_1, \dots, x_n]$.
	\begin{enumerate}[(a)]
		\item For any $i \in [n-1]$,
		\[T_i(PQ) = (\del_iP) \cdot Q + (s_iP) \cdot (T_iQ) .\]
		\item 
		For any permutation $w \in S_n$,
		\[T_w(PQ) = \sum_{v} v(\tdel_{w/v}P) \cdot T_vQ = \sum_{u,v} v(\tdel_{w/v}P) \cdot u(\tdel_{v/u}Q).\]
	\end{enumerate}
\end{prop}
\begin{proof}
	For (a), the right hand side equals
	\[\frac{P-s_iP}{x_i-x_{i+1}} \cdot Q + (s_i P) \cdot \left(s_i Q + \frac{Q - s_i Q}{x_i-x_{i+1}}\right) = \frac{PQ-(s_iP)(s_iQ)}{x_i-x_{i+1}} + (s_iP)(s_iQ).\]
	The first term then equals $\del_i(PQ)$ while the second equals $s_i(PQ)$, and these sum to $T_i(PQ)$, as desired.
	
	For (b), let $s_{i_1} \cdots s_{i_\ell}$ be an expression for $w$. By iteratively applying (a), we get a term containing $T_vQ$ for every subset $J \subset [\ell]$ such that $\prod_{k \in J} T_k = T_v$. The coefficient of $T_vQ$ in this term is then exactly $\varphi_J$, so summing over all $J$ gives $v(\tdel_{w/v}P)$ by definition.
\end{proof}

\subsection{Positivity}
We will show that $\tdel_{w/v}$ can always be expressed as a polynomial in $\del_{ij}$, $i < j$, with positive coefficients by proving the following theorem.

\begin{thm} \label{thm:tdel}
Let $v,w \in S_n$. Choose a reduced expression $w_0v = s_{i_1} \cdots s_{i_\ell}$, and define $\alpha_m$ and $\beta_m$ as in $(*)$. Then
\[\tdel_{w/v} = \sum_J \prod_{j \notin J} \del_{\alpha_j\beta_j},\]
where $J \subset \{1, \dots, \ell\}$ ranges over all subsets such that $\prod_{j \in J} s_{i_j} = w_0w$ (not necessarily reduced).
\end{thm}

This theorem generalizes the analogous positivity result proved in \cite{Liudd} for $\del_{w/v}$, which was the same except that the expression for $w_0w$ must be reduced.

%
\begin{ex} \label{ex:12321pos}
	Let $w = s_1s_2s_3s_2s_1$ and $v=s_1$ as in Example~\ref{ex:12321}. Although the expression for $\tdel_{w/v}$ in that example does not have positive coefficients, using the relation \eqref{triangle} we can rewrite it as
	\begin{align*}
	\tdel_{w/v} &= \del_{23}\del_{34}\del_{23}\del_{12} - \del_{12}\del_{13}\del_{34}\del_{13} + \del_{24}\del_{12} - \del_{12}\del_{14}\\
	&= \del_{23}\del_{34}(\del_{13}\del_{23} + \del_{12}\del_{13}) - (\del_{23}\del_{12} - \del_{13}\del_{23})\del_{34}\del_{13} + \del_{14}\del_{24}\\
	&=\del_{23}\del_{34}\del_{13}\del_{23} + \del_{13}\del_{23}\del_{34}\del_{13} + \del_{14}\del_{24},
	\end{align*}
	which does have positive coefficients.
	
	From Theorem~\ref{thm:tdel}, we can obtain a positive formula more directly using the reduced expression $w_0v = s_1s_2s_3s_1s_2$. As seen in the diagram below, the values of $\del_{\alpha_j\beta_j}$ for $j=1, \dots, 5$ are $\del_{34},\del_{14},\del_{24},\del_{13},\del_{23}$. 
	
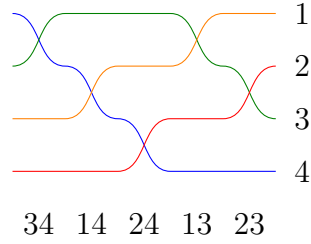
\[
	\begin{tikzpicture}[scale=.7]
	\draw[blue] (0,4) to[out = 0, in=180] (1,3) to[out = 0, in=180] (2,2) to[out = 0, in=180] (3,1)--(5,1);
	\draw[green!50!black] (0,3) to[out = 0, in=180] (1,4) -- (3,4) to[out = 0, in=180] (4,3) to[out = 0, in=180](5,2);
	\draw[orange] (0,2) --(1,2) to[out = 0, in=180] (2,3)--(3,3) to [out = 0, in=180] (4,4)--(5,4);
	\draw[red] (0,1)--(2,1) to[out = 0, in=180] (3,2)--(4,2) to[out = 0, in=180] (5,3);
	\node at (5.5,4) {$1$};
	\node at (5.5,3) {$2$};
	\node at (5.5,2) {$3$};
	\node at (5.5,1) {$4$};
	\node at (.5,0) {$34$};
	\node at (1.5,0) {$14$};
	\node at (2.5,0) {$24$};
	\node at (3.5,0) {$13$};
	\node at (4.5,0) {$23$};
	\end{tikzpicture}
	\]
	Since $w_0w = s_2 = s_1s_1s_2$, we can take $J=\{2\}$, $J=\{5\}$, or $J=\{1,4,5\}$.
	Thus taking the terms $\del_{\alpha_j\beta_j}$ for $j \notin J$ gives
	\[\tdel_{w/v} = \del_{34}\del_{24}\del_{13}\del_{23} + \del_{34}\del_{14}\del_{24}\del_{13} + \del_{14}\del_{24}.\]
	One can check that this formula can be transformed into the previous one by an appropriate application of the relations in Proposition~\ref{prop:relations}.
\end{ex}

The proof of Theorem~\ref{thm:tdel} will follow directly from the following lemma.

\begin{lemma} \label{lemma:tdel}
	Let $v,w \in S_n$ and $i \in [n-1]$ such that $\ell(s_i v)>\ell(v)$. Then
	\[\tdel_{w/v} =  \tdel_{s_iw/s_iv} + \del_{\alpha\beta}\tdel_{w/s_iv},\]
	where $\alpha = v^{-1}(i) < \beta = v^{-1}(i+1)$.
\end{lemma}
\begin{proof}
	Consider any expression $s_iw = s_{i_0}s_{i_1} \cdots s_{i_\ell}$ starting with $s_i = s_{i_0}$. To compute $\tdel_{s_iw/s_iv}$, we must find a subset $J \subset \{0, 1, \dots, \ell\}$ such that $\prod_{j \in J} s_{i_j} = s_iv$ and then use relation \eqref{commute} to move these terms to the left. If $0 \in J$, then $J \setminus \{0\}$ defines a subword of $s_{i_1} \cdots s_{i_\ell} = w$ whose product is $v$, so these terms contribute $\tdel_{w/v}$ to $\tdel_{s_iw/s_iv}$. If $0 \notin J$, then $J$ defines a subword of $s_{i_1} \cdots s_{i_\ell} = w$ whose product is $s_iv$, so these terms contribute $\del_{\beta\alpha}\tdel_{w/s_iv}$, where $\beta = (s_iv)^{-1}(i) = v^{-1}(i+1)$ and $\alpha = (s_iv)^{-1}(i+1) = v^{-1}(i)$. (Note $\alpha < \beta$ since $\ell(s_iv) > \ell(v)$.) It follows that
	\[\tdel_{s_iw/s_iv} = \tdel_{w/v} - \del_{\alpha\beta} \tdel_{w/s_iv}.\]
	Rearranging gives the desired equality.
%
\end{proof}

It is now straightforward to deduce Theorem~\ref{thm:tdel}.

\begin{proof}[Proof of Theorem~\ref{thm:tdel}]
	We induct on $\ell(w_0v)$. When $\ell(w_0v) = 0$, $v=w_0$, so $\tdel_{w/v}$ can only be nonzero when $w=w_0$ as well, in which case $\tdel_{w/v} = 1$, as desired.
	
	Otherwise, choose a reduced expression $w_0v = s_{i_1} \cdots s_{i_\ell}$. Let $i = n-i_1$, so that $s_i = w_0 s_{i_1} w_0$ and $w_0 s_i v = s_{i_2} \cdots s_{i_\ell}$. Then $\ell(s_iv) > \ell(v)$, so by Lemma~\ref{lemma:tdel},
	$\tdel_{w/v} = \tdel_{s_iw/s_iv} + \del_{\alpha\beta}\tdel_{w/s_iv},$
	where 
	\begin{align*}
	\alpha &= v^{-1}(i) = s_{i_\ell} \cdots s_{i_2}s_{i_1} w_0(i) = s_{i_\ell} \cdots s_{i_2}s_{i_1}(i_1+1) = s_{i_\ell}\cdots s_{i_2}(i_1) = \alpha_1,\\
	\beta &= v^{-1}(i+1) = s_{i_\ell} \cdots s_{i_2}s_{i_1} w_0(i+1) = s_{i_\ell} \cdots s_{i_2}s_{i_1}(i_1) = s_{i_\ell}\cdots s_{i_2}(i_1+1) = \beta_1.
	\end{align*}
	By the inductive hypothesis, we have that
	\[\tdel_{w/v} =
	\tdel_{s_iw/s_iv} + \del_{\alpha_1\beta_1}\tdel_{w/s_iv} = \sum_{J'} \prod_{1<j \notin J'} \del_{\alpha_j\beta_j} + \del_{\alpha_1\beta_1}\sum_{J''} \prod_{1<j \notin J''} \del_{\alpha_j\beta_j},\]
	where $J',J'' \subset \{2, \dots, \ell\}$, $\prod_{j \in J'} s_{i_j} = w_0s_iw = s_{i_1}w_0w$, and $\prod_{j \in J''} s_{i_j} = w_0w$. Thus  either $J=\{1\}\cup J'$ or $J = J''$ implies $\prod_{j \in J} s_{i_j} = w_0w$, so the right hand side equals the desired expression $\sum_J \prod_{j \notin J} \del_{\alpha_j\beta_j}$.
\end{proof}

\subsection{Polynomial action}

The operators $\tdel_{w/v}$ act particularly nicely on the \emph{elementary symmetric polynomials} $e_m^{(k)}$ and \emph{complete homogeneous symmetric polynomials} $h_m^{(k)}$.

For any subset $A \subset [n]$, denote
\[
e_m(A) = \sum_{\substack{i_1 < \cdots < i_m\\i_1, \dots, i_m \in A}}x_{i_1} \cdots x_{i_m},\qquad
h_m(A) = \sum_{\substack{i_1 \leq \cdots \leq i_m\\i_1, \dots, i_m \in A}}x_{i_1} \cdots x_{i_m}.
\]
By convention, $e_0(A) = h_0(A) = 1$ and $e_m(A) =h_m(A)= 0$ for $m<0$. We will abbreviate $e_m^{(k)} = e_m(\{1, \dots, k\})$ and $h_m^{(k)} = h_m(\{1, \dots, k\})$.

We then have the following action of divided difference operators.
\begin{lemma} \label{lemma:del-e}
\begin{align*}
\del_{ij} e_m(A) &= \begin{cases}
e_{m-1}(A \setminus \{i\})&\text{if }i \in A, j \notin A,\\
-e_{m-1}(A \setminus \{j\})&\text{if }i \notin A, j \in A,\\
0&\text{otherwise;}
\end{cases}\\
\del_{ij} h_m(A) &= \begin{cases}
h_{m-1}(A \cup \{j\})&\text{if }i \in A, j \notin A,\\
-h_{m-1}(A \cup \{i\})&\text{if }i \notin A, j \in A,\\
0&\text{otherwise.}
\end{cases}
\end{align*}
\end{lemma}
\begin{proof}
	We prove the result for $h_m(A)$, as the proof for $e_m(A)$ is similar. Suppose $i \in A$ and $j \notin A$. Then
	\begin{align*}
	\del_{ij} h_m(A) &=  \sum_{k=0}^m \del_{ij}x_i^k \cdot h_{m-k}(A \setminus \{i\})\\
	&= \sum_{k=1}^m \left(\sum_{l=0}^{k-1} x_i^l x_j^{k-l-1}\right)\cdot h_{m-k}(A \setminus \{i\})\\
	&= h_{m-1}(A \cup \{j\}).
	\end{align*}
	The second case then follows by relation~\eqref{negative}, and the third case follows since $h_m(A)$ will be symmetric in $x_i$ and $x_j$.
\end{proof}

We can then use the expansion given in Theorem~\ref{thm:tdel} to give an explicit description of the action of $\tdel_{w/v}$ on $e_m^{(k)}$ and $h_m^{(k)}$. In order to state the result, we will need the following proposition.

\begin{prop} \label{prop:chain}
	Let $v,w \in S_n$ and fix $1 \leq k < n$. Then, up to reordering commuting transpositions, there is at most one way to write $w = v s_{a_1b_1}\cdots s_{a_tb_t}$ with $a_i \leq k < b_i$ for all $i$ such that $a_1, \dots, a_t$ are distinct (or alternatively, such that $b_1, \dots, b_t$ are distinct). 
	This is possible if and only if each nontrivial cycle of $v^{-1}w$ contains exactly one element larger than $k$ (resp.\ at most $k$).
\end{prop}
\begin{proof}
	Note that if the $a_i$ are distinct, then $s_{a_ib_i}$ and $s_{a_jb_j}$ will commute if $b_i \neq b_j$. Therefore it suffices to show that there is at most one way to write $v^{-1}w = s_{a_1b_1} \cdots s_{a_tb_t}$ with the $a_i$ distinct, $b_1 \leq b_2 \leq \cdots \leq b_t$, and $a_i \leq k < b_i$ for all $i$.
	
	If $b_1 = \cdots = b_r = b$ and $b_{r+1} \neq b$, then $s_{a_1b}s_{a_2b} \cdots s_{a_rb}$ is the cycle $(b\; a_r\; a_{r-1} \; \cdots \; a_1)$. Since the other transpositions cannot involve any of these elements, this cycle must occur in $v^{-1}w$. Note that $b$ is the only element in this cycle larger than $k$, and $a_1, \dots, a_r$ are then uniquely determined by the order of the elements in the cycle. Applying the same logic to the remaining distinct values of $b_i$ and the remaining cycles of $v^{-1}w$ gives the desired result.
	%
\end{proof}

If the expression in Proposition~\ref{prop:chain} with $a_1, \dots, a_t$ distinct exists, and moreover
\[\ell(v) < \ell(vs_{a_1b_1}) < \ell(vs_{a_1b_1}s_{a_2b_2}) < \cdots < \ell(vs_{a_1b_1}s_{a_2b_2} \cdots s_{a_tb_t}) = \ell(w), \tag{$\dagger$}\]
then we will write $A_k(v,w) = \{a_1, \dots, a_t\}$, otherwise we will say that $A_k(v,w)$ does not exist. Equivalently, by the proof of Proposition~\ref{prop:chain}, $A_k(v,w)$ exists if and only if each nontrivial cycle of $v^{-1}w$ has exactly one element larger than $k$, and if for each such cycle $(b\; a_r\; a_{r-1}\; \cdots \; a_1)$ with $b>k$, we have
\[v(b) > v(a_1) > v(a_2) > \cdots > v(a_r),\tag{$\ddagger$}\]
or equivalently,
\[w(a_1) > w(a_2) > \cdots > w(a_r) > w(b).\]

Note that condition $(\ddagger)$ and hence condition $(\dagger)$ are unchanged upon reordering commuting transpositions.

\begin{remark} \label{rmk:ak}
	If $A_k(v,w)$ exists, then it will consist of all elements $a \leq k$ in a nontrivial cycle of $v^{-1}w$, that is, such that $v^{-1}w(a)\neq a$. Hence $[k] \setminus A_k(v,w)$ consists of all $a \leq k$ such that $v(a) = w(a)$.
\end{remark}

Similarly, if the expression in Proposition~\ref{prop:chain} with $b_1, \dots, b_t$ distinct exists along with $(\dagger)$, then we will write $B_k(v,w) = \{b_1, \dots, b_t\}$, otherwise $B_k(v,w)$ will not exist.


We can now state the action of $\tdel_{w/v}$ on $e_m^{(k)}$ and $h_m^{(k)}$.

\begin{thm} \label{thm:tdele}
Let $v,w \in S_n$. Then 
\begin{align*}
\tdel_{w/v} e_m^{(k)} &= \begin{cases}
e_{m-|A|}([k] \setminus A) & \text{if $A=A_k(v,w)$ exists,}\\
0&\text{otherwise};
\end{cases}\\
\tdel_{w/v} h_m^{(k)} &= \begin{cases}
h_{m-|B|}([k] \cup B) & \text{if $B=B_k(v,w)$ exists,}\\
0&\text{otherwise}.
\end{cases}
\end{align*}
\end{thm}
\begin{proof}
We prove the claim for $e_m^{(k)}$, as the proof for $h_m^{(k)}$ is similar.

Consider the parabolic subgroup $S_k \times S_{n-k} \subset S_n$, where $S_k$ acts on the first $k$ letters and $S_{n-k}$ acts on the last $n-k$ letters. Choose a reduced word \[w_0 v = s_{i_1}s_{i_2} \cdots s_{i_p} \cdot s_{i_{p+1}}s_{i_{p+2}} \cdots s_{i_\ell} = v'\cdot u,\] where $k \neq i_{p+1}, \dots, i_\ell$ (so that $u=s_{i_{p+1}} \cdots s_{i_\ell} \in S_k \times S_{n-k}$) and $v'=s_{i_1}\cdots s_{i_p}$ is a minimal length (left) coset representative of $S_k \times S_{n-k}$. Then for $q \leq p$, $\alpha_q \leq k < \beta_q$, while for $q > p$, either $\alpha_q < \beta_q \leq k$ or $k < \alpha_q < \beta_q$.

It follows from Lemma~\ref{lemma:del-e} that $\del_{\alpha_q\beta_q} e_m^{(k)} = 0$ if $q > p$. Thus, in Theorem~\ref{thm:tdel}, the contribution of a subset $J$ to $\tdel_{w/v}e_m^{(k)}$, namely
$\prod_{j \notin J} \del_{\alpha_j\beta_j} e_m^{(k)}$,
 will be zero unless $\{p+1, \dots, \ell\} \subset J$. Moreover, by Lemma~\ref{lemma:del-e}, each application of $\del_{\alpha_j\beta_j}$ for $j \leq p$ decreases the degree by $1$ and removes $x_{\alpha_j}$ from the elementary symmetric function, so we must also have that the $\alpha_j$ for $j \notin J$ are distinct. Thus $J$ contributes $e_{m-|A'|}([k] \setminus A')$, where $A' = \{\alpha_j \mid j \notin J\}$. We therefore need only show that there is a subset $J$ that gives a nonzero contribution if and only if $A_k(v,w)$ exists, in which case $J$ is unique and $A' = A_k(v,w)$.
 
Suppose $A_k(v,w)$ exists, and consider any cycle $(b \; a_r\; \cdots \; a_1) = s_{a_1b}s_{a_2b} \cdots s_{a_rb}$ of $v^{-1}w$. Then $(\ddagger)$ implies $w_0v(b) < w_0v(a_1) < \cdots < w_0v(a_r)$. Since $a_1, \dots, a_r < b$, it follows that $(a_1, b)$, \dots, $(a_r, b)$ must occur as some $(\alpha_j, \beta_j)$ coming from the $v'$ part of $w_0v = v'u$ (that is, with $j \leq p$). Moreover, since $v'$ is a minimal coset representative of $S_k \times S_{n-k}$, it must preserve the order of the $a_i$ as they are all at most $k$. It follows that $(a_1, b)$, \dots, $(a_r,b)$ must occur in that order. Now let $J$ be the set of all $j$ such that $(\alpha_j, \beta_j) \neq (a_i, b)$ for any $i$ and any cycle of $v^{-1}w$. By Proposition~\ref{prop:factor}, \[\prod_{j \in J} s_{i_j} = w_0v \cdot \prod_{j \notin J} s_{\alpha_j\beta_j} = w_0v \cdot v^{-1}w = w_0w.\]
It follows that this set $J$ gives a nonzero contribution to $\tdel_{w/v}e_m^{(k)}$ and that $A' = A_k(v,w)$. Since the order in which the $s_{\alpha_j\beta_j}$ appear is determined by the reduced word chosen for $w_0v$, none of the reorderings of $\prod_{j \notin J}s_{\alpha_j \beta_j}$ can occur, so $J$ is unique. The converse direction is similar.
\end{proof}

As a corollary, we can deduce the Pieri rule for Schubert polynomials.

\begin{cor} \label{cor:pieri}
	Let $v,w \in S_\infty$ with $m = \ell(w) - \ell(v)$. The coefficient of $\mathfrak S_w$ in the Schubert expansion of $\mathfrak S_v \cdot e_m^{(k)}$ is $1$ if $A=A_k(v,w)$ exists with $|A|=m$, and $0$ otherwise.
	
	Similarly, the coefficient of $\mathfrak S_w$ in the Schubert expansion of $\mathfrak S_v \cdot h_m^{(k)}$ is $1$ if $B=B_k(v,w)$ exists with $|B|=m$, and $0$ otherwise.
\end{cor}
\begin{proof}
	By Proposition~\ref{prop:cuvw}, the coefficient of $\mathfrak S_w$ in $\mathfrak S_v \cdot e_m^{(k)}$ is $\del_{w/v} e_m^{(k)}$. But $\del_{w/v}$ is the maximum degree part of $\tdel_{w/v}$, so this coefficient is just the constant term of $\tdel_{w/v} e_m^{(k)}$. The result then follows from Theorem~\ref{thm:tdele} since we must have $|A|=m$ to get a nonzero constant term. The proof for $h_m^{(k)}$ is similar.
\end{proof}
Note that if $|A|=\ell(w)-\ell(v)$, then we must have that in $(\dagger)$, the length goes up by exactly $1$ at each step. This is then easily seen to be equivalent to the phrasing of the Pieri rule in Theorem~\ref{thm:pieri}.

\section{Twisted Schubert polynomials}
In this section, we will apply the twisted operators to define twisted Schubert polynomials. We will then use the results of the previous section to prove that these polynomials are monomial positive as well as give a combinatorial interpretation for their coefficients. We will also define double versions of these polynomials and prove a positivity property for them and their localizations.

\subsection{Definition}

Define \emph{twisted Schubert polynomials} $\tschub_w$ for $w \in S_n$ by 
\[\tschub_w = T_{w^{-1}w_0} (x_1^{n-1} x_2^{n-2} \cdots x_{n-1}).\]
Hence
$\tschub_{w_0} = x_1^{n-1}x_2^{n-2} \cdots x_{n-1}$, and $\tschub_{ws_i} = T_i \tschub_w$. (This holds for all $i$ and $w$ with no restrictions since the $T_i$ satisfy the Coxeter relations.)

It is important to note that $\tschub_w$ does not have the same stability property as $\schub_w$, that is, the value of $\tschub_w$ depends on $n$.

\begin{ex}
	Using the recursion above, we can calculate $\tschub_w$ for all $w \in S_3$.
	\begin{align*}
	\tschub_{s_1s_2s_1} &= x_1^2x_2\\
	\tschub_{s_1s_2} 
	&= x_1x_2^2 + x_1x_2\\
	\tschub_{s_2s_1}
	&= x_1^2x_3 + x_1^2\\
	\tschub_{s_1} 
	&= x_1x_3^2 + x_1x_2 + 2x_1x_3 + x_1\\
	\tschub_{s_2} 
	&=x_2^2x_3 + x_1x_3 + x_2^2 + x_2x_3 + x_1 + x_2\\
	\tschub_{id} 
	&= x_2x_3^2 + x_1x_2 + 2x_2x_3 + x_3^2 + x_2 + 2x_3 + 1
	\end{align*}
	Note that these are all polynomials with positive coefficients. However, the calculation to obtain these polynomials is not inherently positive. For instance, $\tschub_{id} = T_2 \tschub_{s_2} = (s_2 + \del_2)\tschub_{s_2}$, but 
	\[\del_2 \tschub_{s_2} = x_2x_3 -x_1 + x_2 + x_3 + 1\]
	has a negative coefficient.
\end{ex}

\subsection{Monomial positivity}

Using the results of the previous section, it is straightforward to prove that the $\tschub_w$ are all monomial positive.

\begin{thm} \label{thm:positive}
For $w \in S_n$, 
\[\tschub_w = \sum_{u_1, \dots, u_{n-1}} \prod_{i=1}^{n-1} \prod_{\substack{j \leq i\\j \notin A_i}} x_{u_i(j)},\]
where $u_n = w^{-1}w_0$, and the sum ranges over all sequences $u_1, u_2, \dots, u_{n-1}$ such that $A_i = A_i(u_i, u_{i+1})$ exists.

In particular, the polynomial $\tschub_w$ has nonnegative coefficients when expressed in the monomial basis.
\end{thm}
\begin{proof}
Write
\[\tschub_{w_0} = x_1^{n-1}x_2^{n-2} \cdots x_{n-1} = e_{n-1}^{(n-1)} e_{n-2}^{(n-2)} \cdots e_1^{(1)}.\]
By Proposition~\ref{prop:leibniz} (b), we can then express any $\tschub_w = T_{w^{-1}w_0}(\tschub_{w_0}) = T_{w^{-1}w_0}(\prod_{i=1}^{n-1} e_i^{(i)})$ as a sum of products
\[\tschub_w = \sum_{u_1, \dots, u_{n-1}} \prod_{i=1}^{n-1} u_i(\tdel_{u_{i+1}/u_i} e_i^{(i)}).\]
Hence by Theorem~\ref{thm:tdele}, for each sequence $u_1, u_2, \dots, u_{n-1}, u_n = w^{-1}w_0$ for which $A_i=A_i(u_i, u_{i+1})$ exists, we get a contribution of \[\prod_{i=1}^{n-1} u_i(e_{i-|A_i|}([i] \setminus A_i)) = \prod_{i=1}^{n-1} \prod_{\substack{j \leq i\\j \notin A_i}}x_{u_i(j)},\]
as desired.
\end{proof}

Note that by Remark~\ref{rmk:ak}, if $j \leq i$, then $j \notin A_i$ if and only if $u_i(j) = u_{i+1}(j)$.

\begin{ex} \label{ex:123}
	Let $w = 123 \in S_3$, so that $u_3 = w^{-1}w_0 = 321$. There are then nine possibilities for $u_2$ and $u_1$, as shown in the diagram below. Each edge is labeled with the set $\{u_i(j) \mid j \leq i, \; j \notin A_i(u_i, u_{i+1})\}$ (or is left unlabeled if that set is empty).
	
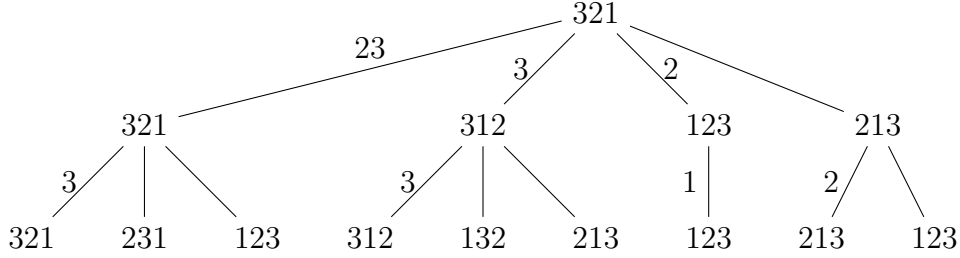
\[
	\begin{tikzpicture}[scale=1.5]
	\node (u3) at (1,0) {$321$};
	\node (u20) at (-3,-1) {$321$};
	\node (u21) at (0,-1) {$312$};
	\node (u22) at (2,-1) {$123$};
	\node (u23) at (3.5,-1) {$213$};
	\node (u10) at (-4,-2) {$321$};
	\node (u11) at (-3,-2) {$231$};
	\node (u12) at (-2,-2) {$123$};
	\node (u13) at (-1,-2) {$312$};
	\node (u14) at (0,-2) {$132$};
	\node (u15) at (1,-2) {$213$};
	\node (u16) at (2,-2) {$123$};
	\node (u17) at (3,-2) {$213$};
	\node (u18) at (4,-2) {$123$};
	\draw (u3)--node[above]{$23$}(u20) (u3)--node[left]{$3$}(u21) (u3)--node[right]{$2$}(u22) (u3)--(u23);
	\draw (u20)--node[left]{$3$}(u10) (u20)--(u11) (u20)--(u12);
	\draw (u21)--node[left]{$3$}(u13) (u21)--(u14) (u21)--(u15);
	\draw (u22)--node[left]{$1$}(u16);
	\draw (u23)--node[left]{$2$}(u17) (u23)--(u18);
	\end{tikzpicture}
	\]
	Summing the contribution for each chain gives
	\begin{align*}
	\tschub_{123} &= x_2x_3\cdot (x_3 + 2) + x_3 \cdot (x_3+2) + x_2 \cdot x_1 + 1 \cdot (x_2 + 1)\\
	&= x_2x_3^2 + 2x_2x_3 + x_3^2 + 2x_3 + x_1x_2 + x_2 + 1.
	\end{align*}
\end{ex}

\subsection{Relation to $\schub_w$} \label{sec:sw}
Since $T_i = s_i + \del_i$, the part of $\tschub_w$ of minimum degree is just $\schub_w$. We can verify that the combinatorial description for $\tschub_w$ in the previous section recovers the known combinatorial descriptions of $\schub_w$.

In order for $u_1, \dots, u_{n-1}, u_n = w^{-1}w_0$ to contribute to $\schub_w$, we must have the $A_i(u_i, u_{i+1})$ be as large as possible. This implies that we must have $u_1 = id$ and $|A_i(u_i, u_{i+1})| = \ell(u_{i+1}) - \ell(u_i)$ for all $i$ (so that in $(\dagger)$, the length must go up by exactly $1$ at each step).

We claim that these conditions imply that $u_i(b) = b$ for $b > i$. Indeed, if the claim holds for $u_i$, then in $(\dagger)$ we must have $\ell(u_i) + 1 = \ell(u_is_{a_1b_1})$. This is only possible if $b_1 = i+1$ (since otherwise $a_1 < i+1 < b_1$ and $u_i(a_1) < u_i(i+1) = i+1 < u_i(b_1) = b_1$). But since we can reorder commuting transpositions without changing the validity of $(\dagger)$, this implies that actually all $b_j$ must equal $i+1$. Hence $u_{i+1}(b) = u_i(b) = b$ for $b>i+1$.

We must therefore have $u_{i+1} = u_{i} \cdot (i+1\;a_r\;a_{r-1}\;\cdots\;a_1)$ with $i+1=u_{i}(i+1) > u_{i}(a_1) > \cdots > u_{i}(a_r)$. In particular, letting $a_j' = u_{i}(a_j)$ so that $i+1 > a_1' > \cdots > a_r'$, we can write this as
\[w_0 u_{i+1}^{-1} = w_0 u_{i}^{-1}(i+1\;a_1'\;\cdots \;a_r').\]

Given a reduced pipe dream for $w$ with pipes labeled $1, \dots, n$ along the top, define $w_0u_i^{-1}$ to be the permutation obtained by reading the order of the pipes down the left side of the $(n+1-i)$th column, followed by $n-i, n-i-1, \dots, 2, 1$. Then $w_0u_i^{-1}$ and $w_0u_{i+1}^{-1}$ are related exactly as dictated by the equation above, where $a_0', \dots, a_r'$ are the locations of the elbows in column $n-i$. Therefore the locations of the crosses in this column correspond to elements of $A_{i}(u_{i},u_{i+1})$. (These are  the rows in which the wires go straight across this column as implied by Remark~\ref{rmk:ak}.) One can also check that $u_n=w^{-1}w_0$, $u_1=id$, and that the length condition on $u_i$ and $u_{i+1}$ is equivalent to the pipe dream being reduced. It follows that this map gives a bijection between pipe dreams and sequences $u_1, \dots, u_n$ contributing to $\schub_w$, as desired.

\begin{ex}
	The pipe dream for $2431$ shown below corresponds to the sequence $u_1, u_2, u_3, u_4$ as shown.
	
			\[
			\vc
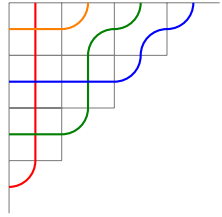
{
				\begin{tikzpicture}[scale=.7]
				\draw[black!50] (0,0) grid (3,-1) (0,-1) grid (2,-2) (0,-2) grid (1,-3) (3,0)--(4,0) (0,-3)--(0,-4);
				\draw[red,thick] (.5,0)--(.5,-3) to[out = -90, in=0] (0,-3.5);
				\draw[orange,thick] (1.5,0) to[out=-90,in=0] (1,-.5)--(0,-.5);
				\draw[green!50!black,thick] (2.5,0) to[out=-90,in=0] (2,-.5) to [out=180, in=90] (1.5,-1)--(1.5,-2) to [out=-90,in=0] (1,-2.5) to (0,-2.5);
				\draw[blue,thick] (3.5,0) to [out=-90,in=0] (3,-.5) to[out=180,in=90] (2.5,-1)to[out=-90,in=0] (2,-1.5)--(0,-1.5);
				\end{tikzpicture}
			}
			\qquad
			\begin{array}{ccc}
			w_0u_4^{-1} = 2431 && u_4=2314\\
			w_0u_3^{-1} = 2431 && u_3=2314\\
			w_0u_2^{-1} = 3421 && u_2=2134\\
			w_0u_1^{-1} = 4321 && u_1=1234 
			\end{array}
			\]

	Note $A_i(u_i,u_{i+1})$, which are the numbers at most $i$ in the same location in both $u_i$ and $u_{i+1}$, give the location of the crosses in column $4-i$.

\end{ex}

\subsection{Double polynomials}

Just as one can define double Schubert polynomials, one can define a ``double version'' of the twisted Schubert polynomials in two sets of variables $x_1, \dots, x_n$ and $y_1, \dots, y_{n}$ by
\[\tschub_w(x,y) = T_{w^{-1}w_0} \prod_{i+j \leq n} (x_i-y_j),\]
where the operators $T_i$ act only on the $x$-variables. Using a similar argument as in Theorem~\ref{thm:positive}, we can give a combinatorial formula for $\tschub(x,y)$.

\begin{thm} \label{thm:doublepos}
	For $w \in S_n$,
	\[\tschub_w(x,y) = \sum_{u_1, \dots, u_{n-1}} \prod_{i=1}^{n-1} \prod_{\substack{j \leq i\\j \notin A_i}} (x_{u_i(j)} - y_{n-i}),\]
	where $u_n = w^{-1}w_0$, and the sum ranges over all sequences $u_1, u_2, \dots, u_{n-1}$ such that $A_i = A_i(u_i, u_{i+1})$ exists.
	
	In particular, $\tschub_w(x,y)$ is a polynomial with nonnegative coefficients in variables $x_i - y_j$.
\end{thm}
\begin{proof}
	Write
	\[\prod_{i+j \leq n}(x_i-y_j) = \prod_{i=1}^{n-1} \prod_{j=1}^{i}(x_j - y_{n-i}).\]
	For fixed $i$, let $x_j' = x_j - y_{n-i}$, so that $\prod_{j=1}^i (x_j-y_{n-i}) = \prod_{j=1}^i x_j'$. For any polynomial $f(x_1, \dots, x_n)$, we have
	$\del_{ab}(f(x_1',\dots, x_n')) = (\del_{ab}f)(x_1', \dots, x_n')$, so Lemma~\ref{lemma:del-e} and Theorem~\ref{thm:tdele} still hold if we replace $x_j$ with $x_j'$ in the definition of $e_m$.
	
	Hence, using Proposition~\ref{prop:leibniz} as in Theorem~\ref{thm:positive}, we get that
	\begin{align*}
	\tschub_w(x,y) &= \sum_{u_1, \dots, u_{n-1}} \prod_{i=1}^{n-1} u_i\left(\tdel_{u_{i+1}/u_i} \prod_{j=1}^{n-i}(x_j-y_{n-i})\right)\\
	&= \sum_{u_1, \dots, u_{n-1}} \prod_{i=1}^{n-1} \prod_{\substack{j \leq i\\j \notin A_i}} (x_{u_i(j)} - y_{n-i}),
	\end{align*}
	as desired.
\end{proof}

\begin{ex} \label{ex:123a}
	Let $w=123 \in S_3$ as in Example~\ref{ex:123}. Using the same diagram as in that earlier example, we replace any $x_j$ coming from an edge label in the top row with $x_j-y_1$ and any $x_j$ from an edge label in the second row with $x_j-y_2$. This gives
	\begin{align*}
	\tschub_{123}(x,y) &= (x_2-y_1)(x_3-y_1) \cdot (x_3-y_2 + 2) + (x_3-y_1) \cdot (x_3-y_2+2)\\
	&\quad{} + (x_2-y_1) \cdot (x_1-y_2) + 1 \cdot (x_2-y_2+1)\\
	&=(1+x_2-y_1)(1+x_3-y_1)(1+x_3-y_2)+(x_2-y_1)(x_1-y_1).
	\end{align*}
\end{ex}


It is easy to check that the bijection described in \S\ref{sec:sw} also explains the connection between the combinatorial formulas for $\tschub_{w}(x,y)$ and $\schub_w(x,y)$.

\subsection{Localization}

Define the localization of $\tschub_v$ at $w$ to be the specialization 
\[\tschub_v(wy,y) = \tschub_v(y_{w(1)}, \dots, y_{w(n)}; y_1, \dots, y_n).\]
In this section, we will give a combinatorial formula for this localization which, as in the ordinary Schubert case, will be a polynomial in $y_b-y_a$, $b>a$, with positive coefficients.

\begin{ex}
	Let $v = 123$. Using the formula for $\tschub_{123}(x,y)$ in Example~\ref{ex:123a}, we can compute the localizations at $w$ for each $w \in S_3$. After some simplification and factorization, we get the following formulas for $\tschub_{123}(wy, y)$:
	\begin{align*}
	w&=321\colon&&1 + (y_2-y_1)(y_3-y_2)\\
	w&=312\colon&&1+y_2-y_1\\
	w&=231\colon&&1+y_3-y_2\\
	w&=213\colon&&(1+y_3-y_1)(1+y_3-y_2)\\
	w&=132\colon&&(1+y_2-y_1)(1+y_3-y_1)\\
	w&=123\colon&&(1+y_2-y_1)(1+y_3-y_1)(1+y_3-y_2)
	\end{align*}
\end{ex}

Note the conspicuous factors of $1+y_b-y_a$ whenever $a<b$ and $w^{-1}(a) < w^{-1}(b)$ in this example. These are precisely the pairs that do not appear as some $(\alpha_j, \beta_j)$ for a given reduced word $w^{-1} = s_{i_1} \cdots s_{i_\ell}$.

We are now ready to state a formula for the localizations of $\schub_v$. (An equivalent formula can also be found in \cite{AMSS, Su}.)

\begin{thm} \label{thm:local}
	Let $v,w \in S_n$, and let $w^{-1} = s_{i_1} \cdots s_{i_\ell}$ be a reduced expression with $\alpha_j$ and $\beta_j$ defined as in $(*)$. Then the localization of $\tschub_v$ at $w$ equals
	\[\tschub_v(wy, y) = \prod_{\substack{1 \leq a < b \leq n\\w^{-1}(a) < w^{-1}(b)}} (1+y_b-y_a) \cdot \sum_J\prod_{j \in J}(y_{\beta_j}-y_{\alpha_j}),\]
	where the sum ranges over all subsets $J \subset [\ell]$ such that $v^{-1} = \prod_{j\in J} s_{i_j}$ (not necessarily reduced).
\end{thm}

Note that we should have that $\schub_v(wy,y)$ is the minimum degree part of $\tschub_v(wy,y)$. Indeed, in this case we can ignore the first product and restrict to the case when $J$ has minimum possible size $\ell(v)$. This formula then immediately reduces to Theorem~\ref{thm:billey}.

To prove this result, we first prove the following lemma which gives a recurrence for these localizations.

\begin{lemma} \label{lem:recurrence}
	For $v,w \in S_n$, $i \in [n-1]$,
	\[(1+y_{w(i+1)}- y_{w(i)})\tschub_v(ws_iy,y) =(y_{w(i+1)}-y_{w(i)})\tschub_{vs_i}(wy,y) +\tschub_v(wy,y).\]
\end{lemma}
\begin{proof}
	We have
	 \[\tschub_{vs_i} = T_i \tschub_v(x,y) = (s_i + \del_i) \tschub_v = \frac{1}{x_i-x_{i+1}} \tschub_v + \left(1 - \frac{1}{x_i-x_{i+1}}\right)s_i\tschub_v.\]
	 Hence plugging in $y_{w(j)}$ for $x_j$ gives
	 \[\tschub_{vs_i}(wy,y) = \frac{1}{y_{w(i)}-y_{w(i+1)}} \tschub_v(wy,y) + \left(1-\frac{1}{y_{w(i)}-y_{w(i+1)}}\right)\tschub_v(ws_iy,y).\]
	 Clearing denominators and rearranging gives the desired result.
\end{proof}

Observe that if all of the localizations of $\tschub_v$ are known, then Lemma~\ref{lem:recurrence} can be used to find all of the localizations of $\tschub_{vs_i}$ (and hence, by iterating, for all twisted Schubert polynomials). Similarly, if all of the localizations at $w$ are known, Lemma~\ref{lem:recurrence} can be used to find all of the localizations at $ws_i$ (and hence, by iterating, at all permutations).

We can now prove the localization formula.

\begin{proof}[Proof of Theorem~\ref{thm:local}]
	Denote \begin{align*}
	P(w) &=  \prod_{\substack{1 \leq a < b \leq n\\w^{-1}(a) < w^{-1}(b)}} (1+y_b-y_a),\\
	Q(v,w) &= \sum_J\prod_{j \in J}(y_{\beta_j}-y_{\alpha_j})
	\end{align*}
	as in the desired expression. It is easy to verify that $Q(v,w)$ does not depend on the reduced expression for $w^{-1}$ (by checking that it is unchanged upon applying the relevant Coxeter relations).
	
	We first check that $\tschub_v(wy,y) = P(w)Q(v,w)$ when $v=w_0$. In this case, note that the only way to have $w(i) \neq j$ for all $i+j \leq n$ is if $w=w_0$. Therefore,
	\[\tschub_{w_0}(wy,y) = \prod_{i+j \leq n} (y_{w(i)}-y_j) = \begin{cases}
	\prod_{a < b} (y_b-y_a) &\text{if }w = w_0,\\
	0&\text{otherwise}.
	\end{cases}\]
	It is easy to check that the only way for the desired formula to be nonzero for $v = w_0$ is if $w=w_0$, in which case $P(w_0)=1$ and $Q(w_0,w_0) = \prod_{a<b} (y_b-y_a)$, as desired.
	
	Suppose $\ell(w) < \ell(ws_i)$. If $w^{-1} = s_{i_1} \cdots s_{i_\ell}$ is a reduced expression, then so is $s_iw^{-1} = s_{i_0}s_{i_1} \cdots s_{i_\ell}$ with $i_0=i$. Since any subexpression for $v^{-1}$ in $s_iw^{-1}$ either contains the initial $s_{i_0}$ or does not, we find that
	\begin{equation} \label{eq:wsi}
	Q(v,ws_i) = (y_{w(i+1)}-y_{w(i)})Q(vs_i,w) + Q(v,w).
	\end{equation}
	Multiplying both sides by $P(w)$ and using the identity $(1+y_{w(i+1)}-y_{w(i)})P(ws_i) = P(w)$ shows that
	\begin{equation} \label{eq:rec1}
	(1+y_{w(i+1)}-y_{w(i)})P(ws_i)Q(v,ws_i) = (y_{w(i+1)}-y_{w(i)})P(w)Q(vs_i,w) + P(w)Q(v,w).
	\end{equation}
	Similarly, replacing $v$ with $vs_i$ in \eqref{eq:wsi} gives
	\begin{equation} \label{eq:wsi2}
	Q(vs_i,ws_i) = (y_{w(i+1)}-y_{w(i)})Q(v,w) + Q(vs_i,w).
	\end{equation}
	Combining \eqref{eq:wsi} and \eqref{eq:wsi2} to eliminate $Q(vs_i,w)$ gives
	\begin{equation}
	(1-(y_{w(i+1)}-y_{w(i)})^2)Q(v,w) = (y_{w(i)}-y_{w(i+1)})Q(vs_i,ws_i) + Q(v,ws_i).
	\end{equation}
	Multiplying both sides by $P(ws_i) = P(w)/(1+y_{w(i+1)}-y_{w(i)})$ then gives
	\begin{equation} \label{eq:rec2}
	(1+y_{w(i)}-y_{w(i+1)})P(w)Q(v,w) = (y_{w(i)}-y_{w(i+1)})P(ws_i)Q(vs_i,ws_i) + P(ws_i)Q(v,ws_i).
	\end{equation}
	
	Together, \eqref{eq:rec1} and \eqref{eq:rec2} imply that $P(w)Q(v,w)$ satisfies the recurrence satisfied by $\tschub_v(wy,y)$ in Lemma~\ref{lem:recurrence} (regardless of whether $\ell(w)$ or $\ell(ws_i)$ is larger), so we must have that $\tschub_v(wy,y) = P(w)Q(v,w)$ everywhere, as desired.
\end{proof}

\section{Conclusion}

In this paper, we have discussed the twisted Schubert polynomials and shown that they have various positivity properties. Still, many questions remain. For instance, the recent results in \cite{AMSS} (see also \cite{Huh, Lee, Stryker}) imply that the class in $H^*(\mathscr F_n) = \CC[x_1, \dots, x_n]/I$ corresponding to $\tschub_w$ for $w \in S_n$ when written in the Schubert basis has coefficients with predictable signs. (This class is, up to an appropriate change of signs, equal to the Chern-Schwarz-MacPherson class of a Schubert cell in the flag variety.) However, a combinatorial interpretation for these coefficients has not yet been described.

It would be interesting to investigate the extent to which the operators $T_i$ preserve monomial positivity. For instance, is it the case that $T_v \schub_w$ is always monomial positive?

%

Finally, in \cite{Kirillov2}, many variants of Schubert, key, and Grothendieck polynomials are considered by varying the operators $T_i$, and various positivity and enumerative properties are explored. Thus it would be interesting to investigate the extent to which the phenomena appearing here generalize to these other variants.

\section{Acknowledgments}
The author would like to thank Alexander Yong for suggesting this line of inquiry and insightful conversations, as well as the organizers of the ``Positivity in Algebraic Combinatorics'' workshop at the Korea Institute for Advanced Study in June 2016.

\bibliography{tschub}
\bibliographystyle{plain}
\end{document}